\documentclass[amsthm]{elsart}

\usepackage{changebar}

\usepackage{yjsco}
\usepackage[numbers]{natbib}

\usepackage{amssymb}
\usepackage{graphics}

\newcommand{\oo}{\ensuremath{\mathcal{O}}}
\newcommand{\kx}{\ensuremath{K[x_1, \dots, x_n]}}

\newcommand{\NN}{\ensuremath{\mathbb{N}^n}}
\newcommand{\ZZ}{\ensuremath{\mathbb{Z}^n}}
\newcommand{\esp}{\ensuremath{\mathcal{E}}}
\newcommand{\abs}{\ensuremath{\mathrm{abs}}}
\newcommand{\V}{\ensuremath{\mathcal{V}}}

\begin{document}
\begin{frontmatter}
\title{Border bases for lattice ideals}

\author{Giandomenico Boffi\corauthref{cor}\thanksref{unintGrant}}
\address{UNINT, Universit\`a degli Studi 
   Internazionali di Roma, via Cristoforo Colombo 200, 00147 Roma.}
\corauth[cor]{Corresponding author.}
\ead{giandomenico.boffi@unint.eu}

\author{Alessandro Logar\corauthref{cor}\thanksref{FRA}}
\address{Dipartimento di Matematica e Geoscienze,
  Universit\`a degli Studi, via Valerio 12/1, 34127 Trieste.}
\ead{logar@units.it}

\thanks[unintGrant]{Partially supported by the UNINT grant 
  ``Metodi relativi allo studio degli ideali polinomiali.''}
\thanks[FRA]{Partially supported by the FRA 2013 grant 
  ``Geometria e topologia delle variet\`{a}'', Universit\`{a} 
  di Trieste, and by the PRIN 2010-2011 grant 
  ``Geometria delle variet\`a algebriche''.}

\begin{abstract}
The main ingredient to construct an $\oo$-border basis of an ideal 
$I\subseteq \kx$ is the order ideal $\oo$, which is a basis of the 
$K$-vector space $\kx/I$. In this paper we give a procedure to find 
all the possible order ideals associated with a lattice ideal $I_M$ (where 
$M$ is a lattice of $\ZZ$). The construction can be applied
to ideals of any dimension (not only zero-dimensional) and shows that 
the possible order ideals are always in a finite number. 
For lattice ideals of positive 
dimension we also show that, although a border basis is infinite, it can be 
defined in finite terms. Furthermore we give an example which proves that
not all border bases of a lattice ideal come from Gr\"obner bases. Finally, 
we give a complete and explicit description of all the border bases for
ideals $I_M$ in case $M$ is a 2-dimensional lattice contained in 
$\mathbb{Z}^2$.
\end{abstract}

\end{frontmatter}

\section{Introduction}
Let $I$ be a zero-dimensional ideal in the polynomial ring $\kx$, then a 
\emph{border basis} of $I$ is composed by a finite set $\oo$ of monomials 
closed under division, which is a basis 
of the $K$-vector space $\kx/I$ and a set of polynomials 
$f_1, \dots, f_m\in I$, such that $f_i = b_i -\sum_{j} a_{ij} t_j$, where 
$t_j \in \oo$, $a_{ij}\in K$ and $b_i$ are elements in the \emph{border} 
of $\oo$ (i.e.\ are not in $\oo$, but are obtained multiplying an 
element of $\oo$ by a variable). Border bases are a natural generalization 
of Gr\"obner bases, and indeed, given a Gr\"obner basis $G$, it is easy to
construct the corresponding border basis (the set $\oo$ is the set of 
irreducible monomials w.r.t.\ $G$). Border bases were introduced 
in~\cite{mmm2} (see also~\cite{mmm1}); for a discussion of their properties, 
see, among others,~\cite{kk1, kk2, kr, mora, m1}. One should notice that 
in fact~\cite{m1} deals with a more general notion of border basis.

The set we have denoted by $\oo$ is often called an order ideal in books 
of commutative algebra, and we use this name throughout. But one can find in 
the literature at least ten other ways of naming $\oo$: see page 6 
of~\cite{kr}, where the alternative terminologies are linked to different
branches of Mathematics.

The main difference between border
bases and Gr\"obner bases lies in the order ideal 
$\oo$ associated with them, which, for border bases, has less constraints, 
since it is not linked to a term order. One specific application of 
border bases regards the determination
of the solutions of a system of polynomial equations in which the coefficients 
are real numbers, known with some approximation: it turns out that border 
bases are more stable under small perturbations of the coefficients than 
Gr\"obner bases and allow therefore to construct better
values for the zeros~(\cite{aft, kr, mt1}). In this paper, however, 
we do not focus our attention on the problem of determining zeros of 
systems of polynomials, but we consider a different question:
we want to find all the border bases
of a given lattice ideal.

A \emph{lattice ideal} is an ideal in the 
polynomial ring which comes from a lattice $M$ in $\ZZ$. More precisely, 
the ideal is generated by the binomials $x^{a^+}-x^{a^-}$, where  
$a = (a_1, \dots, a_n) \in M$ and $a = a^+-a^-$ where $a^+$ is the $n$-tuple 
whose $i$-th element is $a_i$, if $a_i$ is positive and $0$ otherwise
(a similar definition for $a^-$). Lattice ideals arise in many different 
examples: all toric ideals, for instance, are lattice ideals as well as
the ideal associated with an integer programming problem. The construction 
of Gr\"obner bases for lattice ideals was studied by many authors
(see~\cite{blr, bl, bl2, swz} and the references 
given there) and there are efficient symbolic computation packages 
which allow their computation~\cite{4ti2, CoCoA-5, Normaliz}.

In this paper, as stated, we consider an ideal $I_M$ defined by a 
lattice $M\subseteq \ZZ$ ($M$ can therefore equivalently be seen as 
a sub-module of $\ZZ$) and we show how to construct \emph{all} the 
possible border bases of $I_M$. We omit a very strong 
condition that is usually considered for border bases, that is, 
we do not assume that the ideal $I_M$ is zero-dimensional. (For 
another paper in which the positive dimension case for border bases
is considered, see~\cite{mt2}). Let us
remark that the main step in getting a border basis is to find 
an order ideal which is a $K$-basis for $\kx/I_M$ and this problem 
can be converted into the problem of determining an order ideal $\oo$
of $\NN$ (w.r.t.\ the partial order $\preceq$, where $u \preceq v$ if every
component of $u$ is less than or equal to the corresponding component 
of $v$) whose elements uniquely represent $\ZZ/M$. Section~\ref{sez2}
therefore deals with order ideals in $\kx$ and in $\NN$; an order ideal 
of $\NN$ satisfying the above properties will be called a max-compatible
order ideal (w.r.t.\ $M$); later, in section~\ref{coi},
we consider the problem of finding, for a given module $M
\subseteq \ZZ$, all the possible max-compatible order ideals w.r.t.\ $M$.
The construction we propose determines a
finite graph whose maximal cliques allow to recover the required 
order ideals. In particular, in this way we see that there are only 
finitely many order ideals and therefore the ideal $I_M$ has only finitely
many border bases. 

Section~\ref{bb} deals with the case of border bases of $I_M$ 
and in particular we briefly consider the case of infinite border 
bases: the specific shape of any order ideal $\oo$, together with 
the properties of the lattice $M$ allow us to describe both
the border of $\oo$ and the $\oo$-border basis in finite terms. 
Moreover we give 
some examples and in particular we show that there exist border bases 
for lattice ideals which cannot be obtained from any Gr\"obner basis of
that ideal.

The final section considers the very peculiar case in which $M$ is a 
module of rank~2 contained in $\mathbb{Z}^2$. We show that in this case
every border basis comes from a Gr\"obner basis and we see that the results
obtained in the previous sections allow a complete and
explicit description of all the border (Gr\"obner) bases of~$I_M$.

\section{Order ideals}
\label{sez2}
Recall that the commutative monoid $\mathbb{T}$ of the terms of $\kx$
(w.r.t.\ the product) is isomorphic to the additive monoid $\NN$.
If $t \in \mathbb{T}$ the corresponding element of $\mathbb{N}^n$ is denoted
by $\lg(t)$; if $u \in \NN$, the corresponding element of 
$\mathbb{T}$ is denoted by $\esp(u)$ (however, if there is no 
risk of ambiguity, sometimes we will omit the function $\esp$). 
By $e_1, \dots, e_n$ we denote the 
canonical basis of $\ZZ$. On $\NN$ we consider a partial order $\preceq$
given by $u, v \in \NN$, $u \preceq v$ if every component of $u$ is
not bigger than the corresponding component of $v$. If $u \in \NN$, let
\[
D(u) = \{ v \in \NN \mid v \preceq u \}, 
\qquad
C(u) = \{ v \in \NN \mid u \preceq v \}. 
\]
$D(u)\subseteq \NN$ corresponds to the monomials which divide
$\esp(u)$, while $C(u)$ corresponds to the monomials which are 
divided by $\esp(u)$.

If $u, v \in \NN$, $\mathrm{lcm}(u, v)$ is the element
$(\max(u_1, v_1), \dots, \max(u_n, v_n))$ where $u_i$ and $v_i$ are the 
components of $u$ and $v$ respectively. 

An \emph{order ideal} $\oo$ (in $\NN$) is a subset of $\NN$ 
such that, if $u\in \oo$, then $D(u) \subseteq \oo$. 
The \emph{border}
of $\oo$ is the set of elements $u \in \NN$ such that $u \not\in \oo$, 
but there exists $i \in \{1, \dots, n\}$ such that 
$u-e_i \in \oo$. The border of $\oo$ is denoted by $\partial \oo$. Using the
function $\esp$, we can define order ideals in $\kx$: an \emph{order ideal}
is a subset of $\mathbb{T}$ which contains all the divisors of its elements.
Analogously, the \emph{border} of an order ideal of $\kx$ is the set of 
terms which are not in the order ideal, but are obtained multiplying an 
element of the order ideal by one of the variables. We do not
require the finiteness condition of the order ideals. However, if the order
ideal is finite, we do get the usual definition given, for instance,
in~\cite{kr}. Also the definition of the border basis of an ideal $I$
of $\kx$ as given in~\cite{kr} can immediately be extended
to the case in which the order ideal is infinite (hence $I$ is not zero 
dimensional). Clearly, in this case, the border basis is infinite.

Suppose $M \subseteq \ZZ$ is any sub-module of $\ZZ$ of dimension 
$m\leq n$ and assume it is generated by the rows of the following matrix:
\begin{equation}
\left(
\begin{array}{ccccccc}
d_1 & * &  \dots & * & * &\dots&*\\
0   & d_2 & \dots & * & * &\dots&*\\
\dots &  &        &    &  &  &\\
0 & 0 & \dots     & d_m & * &\dots&*
\end{array}
\right)
\label{matr1}
\end{equation}
where $d_1, \dots, d_m$ are positive integers and every ``$*$'' above 
a $d_j$ represents a non-negative integer smaller that $d_j$,
i.e.\ the matrix is in Hermite Normal Form (HNF). 
We associate with $M$ the following subset of $\ZZ$:
\[
B = \{(i_1, \dots, i_n) \mid 0 \leq i_j < d_j \mbox{ for } 
j=1, \dots, m, \ i_j \in \mathbb{Z} \mbox{ for } j=m+1, \dots, n \}.
\]
Note that if $b_1, b_2 \in B$ and $b_1 \equiv_M b_2$ (where 
``$\equiv_M$'' means $b_1-b_2 \in M$), then $b_1 = b_2$ and every element 
of $\ZZ$ has a unique representative, $\mathrm{mod}\, M$, in $B$. Hence the 
elements of $B$ are in one to one correspondence with the elements of the
module $\ZZ/M$.

Given 
$b=(b_1, \dots, b_n)\in \ZZ$, the construction of 
its representative 
($\mathrm{mod}\, M$) $\rho(b)$ in $B$ 
is easily obtained by repeated divisions as follows: 
suppose $b_1, \dots, b_{r-1}$ ($r \leq m$) are 
such that $0 \leq b_i < d_i$ for $i \leq r-1$ and $b_{r}\geq d_{r}$, 
and replace $b$ by $b' = b-qM_{r}$ where $M_{r}$ is the 
$r^{th}$-row of 
the matrix~(\ref{matr1}) and $q$ is the quotient of $b_{r}$ when divided by 
$d_{r}$. Then $b_{r}'$ is such that $0\leq b_{r}' < d_{r}$.

The set $B$ is finite if and only if $m = n$. In this case $B$ is an 
order ideal of $\NN$ and has $d_1d_2\cdots d_n$ elements.
Sometimes it will be convenient to label 
its elements with $0, 1, \dots, d_1d_2\cdots d_n-1 $ in the following way: 
if $(a_1, \dots, a_n)\in B$, then its label is 
\begin{eqnarray}
a_n+d_na_{n-1}+d_nd_{n-1}a_{n-2}+ \cdots + d_n\cdots d_2 a_1.
\label{numB}
\end{eqnarray}
Consequently we can label all the elements of $\ZZ$, assigning the 
same number to equivalent elements. 

Summarizing the properties of the set $B$ when $m=n$, we have: it is an
order ideal; if $b_1, b_2 \in B$ are equivalent $\mathrm{mod}\, M$,
then $b_1=b_2$; $B$ is maximal w.r.t.\ this property and every element
of $\ZZ$ has an equivalent element in $B$. We capitalize on these
properties in the following definition concerning any order ideal of $\NN$
(finite or infinite):

\begin{defn}
Let $\oo$ be an order ideal of $\NN$. Then it is \emph{compatible} 
($\mathrm{mod}\, M$)
if it holds: $a, b \in \oo$ and $a, b$ equivalent $\mathrm{mod}\, M$,
then $a = b$. The order ideal is \emph{maximal compatible} 
($\mathrm{mod}\, M$) if it is compatible
and maximal in the set of compatible order ideals, w.r.t.\ inclusion.
It is \emph{max-compatible} ($\mathrm{mod}\, M$) if every element of 
$\ZZ$ has a representative ($\mathrm{mod}\, M$) in it.
\label{tipiOI}
\end{defn}
Clearly, max-compatible implies maximal compatible. If $\oo$ is finite, 
then max-compatible is equivalent to \emph{maximum compatible}, i.e.\ 
compatible with the maximum number of elements.

Consider now the lattice ideal $I_M$ associated with the module $M$ and 
suppose $m = n$. Then 
$\esp(B)$ is an order ideal of $\kx$ and every term $t\in \mathbb{T}$ 
is equivalent, modulo the ideal $I_M$, to an element 
$\mathcal{R}(t)\in \esp(B)$ defined by $\mathcal{R}(t) = \esp(\rho(\lg(t)))$.
In particular
$t-\mathcal{R}(t)$ is a binomial in $I_M$ and the set:
\[
  \{ 
    u-\mathcal{R}(u) \mid \ \mbox{for } u \in \partial(\esp(B))
  \}
\]
is a first example of a border basis of $I_M$. 

\begin{exmp} 
\label{ex1}
As a particular case, consider in $\mathbb{Z}^2$ the subgroup $M$ 
generated by the rows of the matrix 
\[
\left(
\begin{array}{cc}
2 & 6 \\
0 & 10
\end{array}
\right)
\]
which is in HNF. The set $B\subseteq \mathbb{Z}^2$ is 
$\{(i, j) \ |  \ 0 \leq i < 2, \ 0 \leq j < 10 \}$, hence $\esp(B)$ is 
the set of monomials $\{x^iy^j \ | \ 0 \leq i < 2, \ 0 \leq j < 10 \}$, 
the border (of $\esp(B)$) is 
$\{x^2y^j \ | \ j = 0,\dots, 9\} \cup \{y^{10}, xy^{10}  \}$ and the 
corresponding border basis is:
\[
\{x^2y^i-y^{4+i} \; | \; i = 0, \dots, 5  \} 
\cup \{x^2y^{6+j}-y^j \; | \; j = 0, \dots, 3 \}
\cup \{x^ky^{10}-x^k \; | \; k = 0, 1  \}.
\]
\end{exmp}

\section{Construction of order ideals}
\label{coi}
As usual, $M$ denotes a sub-module of $\ZZ$ of rank $m \leq n$. 
Let $\V$ be the union of all compatible order ideals of $\NN$
(i.e., as said, order ideals which do not contain equivalent elements
$\mathrm{mod}\, M$). If $a \in \ZZ$, 
then $\abs(a)$ denotes $a^+ + a^- \in \NN$. 
\begin{prop}
Let $\mathcal{A} =\{\abs(a)\ | \ a\in M\setminus \{0\}\}$ and 
$\mathcal{A}_1$ be the set of elements of $\mathcal{A}$ which are 
minimal w.r.t.\ the partial order $\preceq$. Then it holds: 
\[
\V = \NN \setminus \bigcup_{a \in \mathcal{A}} C(a) = \NN \setminus 
\bigcup_{a \in \mathcal{A}_1} C(a).
\]
\label{insiemeV}
\end{prop}
\begin{proof}
It is clear that 
$\bigcup_{a \in \mathcal{A}} C(a) = \bigcup_{a \in \mathcal{A}_1} C(a)$, so 
it suffices to prove the first equality. Let $v\in \V$
and suppose there exists 
$\abs(a)\in \mathcal{A}$ such that $v\in C(\abs(a))$, so $\abs(a) \in D(v)$. 
Since  $a = a^+ - a^-\equiv_M 0$, then $a^+ \equiv_M a^-$ and since 
$a^+, a^- \in D(\abs(a)) \subseteq D(v)$, we have that $D(v)$ is not 
compatible. Conversely, let $v \in  \NN \setminus \bigcup_{a \in \mathcal{A}} C(a)$
and suppose $D(v)$ is not compatible, hence there exist $u_1, u_2 \in D(v)$ 
such that $u_1 \equiv_M u_2$. If  $a = u_1-u_2 \equiv_M 0$, then $a^+ \in D(u_1)$
and $a^- \in D(u_2)$, so $\abs(a) = \mathrm{lcm}(a^+, a^-) \in D(v)$, which 
gives that $v \in C(\abs(a))$. 
\end{proof}

\begin{rem}
As a consequence of the above proposition, we see that $\esp(\V)$
is the normal basis of the monomial ideal 
$J = \left( \esp(\abs(a)) \ | \ a \in M\setminus \{0\}\right)$ 
(i.e.\ $\esp(\V)$ is a
$K$-basis of $\kx/J$). Moreover Dixon's lemma 
(see e.g.\ \cite[page~38]{mora})
ensures that $\mathcal{A}_1$ is finite. 
\label{oss1}
\end{rem}

\begin{prop} It holds:
$\mathrm{rank}\, (M) = n$ if and only if the set $\V$ is finite.
\label{tappi} 
\end{prop}
\begin{proof}
If $\mathrm{rank}\, (M) = n$, then for every $i \in \{1, \dots, n\}$ 
we can find an element $t_ie_i \in M$ (where $t_i \in \mathbb{N}$ and 
$e_i$ is an element of the canonical basis of $\ZZ$). If $\oo$ is a 
compatible order ideal, then necessarily 
$\oo \subseteq  D(t_1, \dots, t_n)$. 
Therefore $\V \subseteq D(t_1, \dots, t_n)$ and is a finite set. If 
$\mathrm{rank}\, (M) < n$, then there exists $i \in \{1, \dots, n\}$ such
that $te_i \not \in M$ for all $t \in \mathbb{N}$, hence $\oo = \{
te_i \mid t \in \mathbb{N} \}$ is an infinite compatible order ideal, 
hence $\V$ is infinite. 
\end{proof}

\begin{exmp} We consider again example~\ref{ex1}, i.e.\ the module in 
$\mathbb{Z}^2$ generated by $(2, 6)$ and $(0, 10)$. Since the rank of $M$ 
is~2, the set $\V$ is finite. The set $\mathcal{A}$ is given by the 
blue dots in figure~\ref{figA}, left, the set $\mathcal{A}_1$ is the set
$\{(10, 0),\ (4, 2),\ (2, 4),\ (0, 10) \}$ and the 
set $\V$ is the shadow region in figure~\ref{figA}, right. 
\end{exmp}

\begin{figure}
\resizebox{!}{5cm}{
\includegraphics{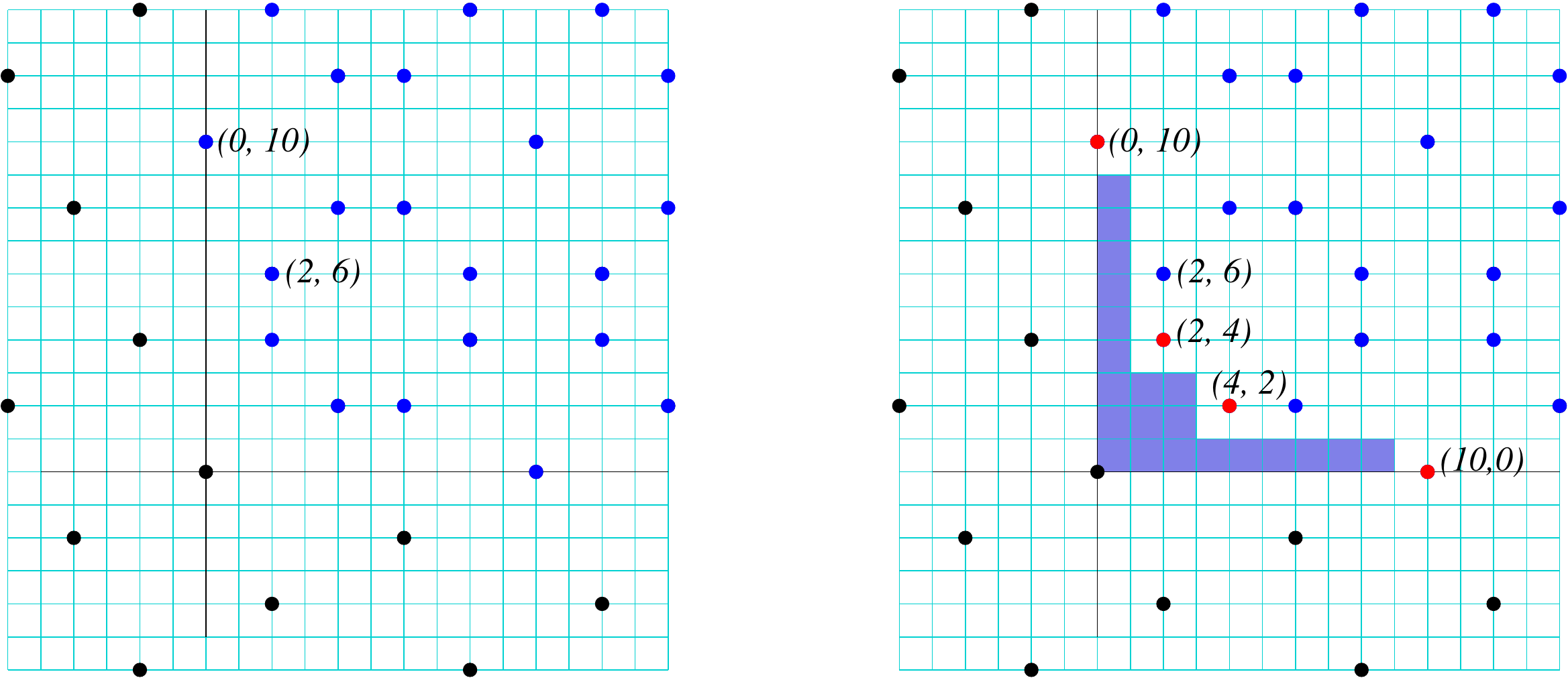}
}
\caption{On the left: the lattice generated by $(2, 6)$ and $(0, 10)$
and the set $\mathcal{A}$ (in blue); on the right: the elements 
of $\mathcal{A}_1$ (in red) and the set $\V$ (the shadow region).}
\label{figA}
\end{figure}

Let $G_\V$ be the graph whose vertexes are the elements of $\V$ and 
two vertexes $u$ and $v$ are connected by an edge if in the 
set $D(u) \cup D(v)$ there are no equivalent elements ($\mathrm{mod}\, M$).
Suppose first that $M$ has rank $m=n$. Then in this case 
$G_\V$ is a finite graph
and a maximal compatible order ideal (see 
definition~\ref{tipiOI}) corresponds to 
a \emph{maximal clique} of $G_\V$ and a maximum compatible order 
ideal (i.e.\ max-compatible) corresponds to a \emph{maximum clique} 
of $G_\V$. Therefore the 
problem of finding maximal and maximum compatible order ideals is 
reduced, at least in the case of rank $n$, to the problem of finding
maximal and maximum cliques of a (finite) graph, which can be done by the 
Bron-Kerbosch algorithm (see~\cite{bk},~\cite{ck}) (an implementation
of the Bron-Kerbosh algorithm can be found, for instance, in~\cite{sage}). 
If the rank of $M$ is less than $n$, the maximal cliques again give the
maximal compatible order ideals, but $G_\V$ is an infinite graph. 
We shall now see how to overcome this problem.

Set $P = \{(k_1, \dots, k_n)\in \NN \ \mid \ \forall j\  k_j \not= 0 \}$
and $-P = \{-k \ \mid \ k \in P\}$. We consider the following set:
\[
\mathcal{X} = 
\left\{
(c^+, c^-) \mid c \in M \setminus \left(P\cup -P \cup \{0\}\right)
\right\} \subseteq \NN \times \NN.
\]
(Note that $\mathcal{X}$ could also be defined as the set of
$(c^-, c^+)$ since, if $c\in M \setminus \left(P\cup -P \cup \{0\}\right)$, 
then also $-c\in M \setminus \left(P\cup -P \cup \{0\}\right)$).
On $\mathcal{X}$ we define a partial order $\sqsubseteq$ given by: 
\[
\mbox{if $a, b \in \mathcal{X}$, \ }
a \sqsubseteq b \mbox{\ \ if\ \ } a_0 \preceq b_0 \mbox{ and } 
a_1 \preceq b_1, \mbox { or } a_0 \preceq b_1 \mbox{ and } 
a_1 \preceq b_0
\]
(where $a_0$ and $a_1$ denote the two coordinates of $a \in \mathcal{X}$).
By $\mathcal{X}_1$ we denote the set of the minimal elements of 
$\mathcal{X}$ w.r.t.\ $\sqsubseteq$.
\begin{prop}
The partial order $\sqsubseteq$ is a well-founded order and 
the set $\mathcal{X}_1$ is finite.
\label{x1}
\end{prop}

\begin{proof} The partial order $\sqsubseteq$ is well-founded since $\preceq$
is well-founded. To see that $\mathcal{X}_1$ is finite, let 
$\epsilon = (\epsilon_1, \dots, \epsilon_n)$ be such that each 
$\epsilon_i \in \{-1, 1\}$ (it is convenient to consider each
$\epsilon$ as an identifier of an orthant of 
$\ZZ$) and let $M_\epsilon = \{\epsilon x  \mid x \in M\}$, where 
$\epsilon x = (\epsilon_1x_1, \dots, \epsilon_nx_n)$.
The module $M_\epsilon$ is constructed in such a way
that the part of $M_\epsilon$ contained in the positive orthant 
corresponds to the part of $M$ contained in the 
orthant in which the signs of the coordinates are given by the 
vector $\epsilon$. Let $B_\epsilon$ be the Hilbert basis of
$(M_\epsilon \setminus \{0\})\cap \NN$ (w.r.t.\ the partial order~$\preceq$)
(see e.g.~\cite[6.1.B]{kr}). 
In particular, $B_\epsilon$ is a finite set and for every element $u$ of 
$(M_\epsilon \setminus \{0\})\cap \NN$ there exists an element $b\in B_\epsilon$ 
such that $b \preceq u$. Let $c\in M \setminus 
\left(P\cup -P \cup \{0\}\right)$ and assume the orthant of $c\;(=c^+-c^-)$ 
is given by the vector $\epsilon$. Then there exists $b\in B_\epsilon$ 
such that $b \preceq \epsilon c=c^++c^-$, hence $(\epsilon b)^+ \preceq c^+$
and $(\epsilon b)^- \preceq c^-$. {}From this it follows that 
$(c^+, c^-)\in \mathcal{X}_1 \Rightarrow c\in \cup_\epsilon \epsilon B_\epsilon$
and hence $\mathcal{X}_1$ is finite.  
\end{proof}

\begin{prop} Let $u, v \in G_{\V}$. The following are equivalent:
\begin{enumerate}
\item $u$ and $v$ are not connected;
\item there exists $a \in \mathcal{X}$ such that 
$a_0 \in D(u)$ and $a_1 \in D(v)$;
\label{due}
\item there exists $a \in \mathcal{X}_1$ such that 
$a_0 \in D(u)$ and $a_1 \in D(v)$.
\label{tre}
\end{enumerate} 
\label{nonConnessione}
\end{prop}

\begin{proof} Suppose $u$ and $v$ are not connected, hence there exist
$u_1\in D(u)$ and $v_1\in D(v)$ such that $u_1 \equiv_M v_1$. 
Then $c = u_1 - v_1$
is an element of $M$ such that $c^+ \in D(u)$ and $c^- \in D(v)$. Moreover, 
$c^+\not =0$, and $c^-\not=0$ (if, for instance, $c^+ = 0$, then 
$c^- \equiv_M 0$ and this is a contradiction, since $v \in \V$), therefore 
$a = (c^+, c^-)\in \mathcal{X}$. If~\ref{due}.\
holds, let $b\in \mathcal{X}$ be such that \ $b \sqsubseteq a$. Then either
$b_0 \preceq a_0$ and $b_1 \preceq a_1$ (hence $b_0 \in D(u)$ and $b_1 \in
D(v)$), or $b_0 \preceq a_1$ and $b_1 \preceq a_0$ and in this case it is enough
to consider $\beta = -b$. Then $\beta \sqsubseteq a$ and $\beta_0 \in
D(u)$ and $\beta_1 \in D(v)$. {}From this~\ref{tre}.\ follows. Finally, 
if~\ref{tre} holds, then $a_0 \equiv_M a_1$ and $u$ and $v$ are not connected.
\end{proof}

We now define a partition on $\V$ (hence on the vertexes of $G_\V$)
as follows: if $u \in \V$, then we set
\begin{eqnarray}
R_u = \{ v \in \V \ | \ \mbox{for all } a \in \mathcal{X}_1, \ a_0 \in D(v) 
\mbox{ iff } a_0 \in D(u) \}.
\label{partizione}
\end{eqnarray}
Again, let us remark that $R_u$ can also be defined by:
\begin{eqnarray*}
R_u = \{ v \in \V \ | \ \mbox{for all } a \in \mathcal{X}_1, \ a_1 \in D(v) 
\mbox{ iff } a_1 \in D(u) \}
\end{eqnarray*}
since $(a_0, a_1) \in \mathcal{X}_1$ if and only if 
$(a_1, a_0) \in \mathcal{X}_1$.\\
We have
\begin{prop}
If $u_1, u_2 \in R_u$, then $u_1$ and $u_2$ are connected. Moreover, 
$u, v \in G_{\V}$ are connected if and only if every element of $R_u$ 
is connected to every element of $R_v$.
\label{connessioni}
\end{prop}

\begin{proof} Suppose $u_1$ and $u_2$ are not connected. Then, by 
proposition~\ref{nonConnessione}, 
there exists $a\in \mathcal{X}$ (minimal) such that $a_0\in D(u_1)$, 
$a_1 \in D(u_2)$, so $a_0 \in D(u)$ and, analogously, $a_1 \in D(u)$, but 
this gives a contradiction, since $u \in \V$ and $a_0 \equiv_M a_1$.
Suppose now that $u$ and $v$ are not connected
and let $u_1 \in R_u$ and $v_1 \in R_v$. Hence
there exists $a \in \mathcal{X}$ such that $a_0 \in D(u)$ and $a_1 \in D(v)$, 
so $a_0 \in D(u_1)$ and $a_1 \in D(v_1)$, hence $u_1$ and $v_1$ are 
not connected. If $u_1 \in R_u$ and $v_1 \in R_v$ are not connected, 
a similar argument shows that $u$ and $v$ are not connected.
\end{proof}

Consider the set $\mathcal{Y} = \{a_0 \mid a\in \mathcal{X}_1 \}\cup \{0\}\, 
(= \{a_1 \mid a\in \mathcal{X}_1 \}\cup\{0\})$, and let 
$c_1, \dots, c_l\in \NN$ be 
the elements of $\mathcal{Y}$. With every 
element $u \in \V$ we associate the $l$-tuple $s(u) = (\chi(c_i, u) \mid 
i=1,\dots, l)$, where 
\[
\chi(c_i, u) = \left\{
\begin{array}{ll}
0 & \mbox{if $c_i \not \in D(u)$}\\
1 & \mbox{if $c_i \in D(u)$}
\end{array}.
\right.
\]
\begin{prop}
For each $u\in \V$, we have:
\[
R_u = \{ v \in \V \mid s(u) = s(v) \}.
\]
\label{partizOperativa}
\end{prop}

\begin{proof} Immediate. \end{proof}
\begin{cor}
The set $\{R_u \mid u \in \V\}$ is a finite set.
\label{partFinita}
\end{cor}

\begin{proof}
The $l$-tuples $s(u)$ can only assume a finite number of values. 
\end{proof}

We recall that a hyper-rectangle of $\NN$ is a set of points 
$(a_1, \dots, a_n) \in \NN$, such that each coordinate $a_i$ is subject
to a condition of the form $l_i \leq a_i < L_i$ where $l_i \in \mathbb{N}$
and $L_i\in \mathbb{N}\cup \{+\infty\}$. 
\begin{prop}
Each set $R_u$ is a finite union of hyper-rectangles.
\label{hyperrect}
\end{prop}

\begin{proof} Let $v \in R_u$ and $c_i \in \mathcal{Y}$. If 
$c_i \in D(u)$, then $c_i \in D(v)$, so each coordinate $v_j$ of 
$v$ is such that $c_{ij} \leq v_j$ (where $c_{ij}$ is the $j$-th coordinate 
of $c_i$). If $c_i \not\in D(u)$, then $c_i \not\in D(v)$, so 
there exists a $k$ such that $v_k < c_{ik}$. Moreover, $v\in \V$ and 
proposition~\ref{insiemeV} give some further upper bounds for the coordinates 
of $v$. Considering all these bounds we see that the elements 
of $R_u$ are subject to a finite number of conditions each of which 
defines a hyper-rectangle. 
\end{proof}

Starting from the partition $R_u, u\in \V$ of the vertexes of $G_\V$, we 
can construct the quotient graph $\widetilde{G}_\V$ whose vertexes are the 
elements of the partition and two vertexes $R_u$ and $R_v$ of 
$\widetilde{G}_\V$ are connected 
if and only if $u$ and $v$ are connected in $G_\V$. According to 
proposition~\ref{connessioni}, the connection is well-defined; moreover, by
corollary~\ref{partFinita}, $\widetilde{G}_\V$ is a finite graph. 

A clique of $G_\V$ is a compatible order ideal. A 
maximal clique of $G_\V$ is a maximal, compatible order ideal and gives a 
maximal clique of $\widetilde{G}_\V$. Conversely, from 
a maximal clique of $\widetilde{G}_\V$, taking the union of its vertexes
(considered as sets), we get a maximal clique of $G_\V$ which
is a maximal, compatible order ideal. Hence, the above constructions
allow us to obtain all the maximal compatible order ideals associated with
a module $M$. 

In a maximal compatible order ideal, by definition, all the
elements are different $\mathrm{mod}\; M$, but it is not true, in general, 
that a maximal compatible order ideal contains a representative of all the
elements of $\ZZ/M$, in other words, not all maximal compatible order 
ideals are also max-compatible (as defined in section~\ref{sez2}).
However among the maximal compatible 
order ideals there are all the max-compatible ones.
The finiteness of $\widetilde{G}_\V$ then yields:
\begin{prop}
Given a module $M \subseteq \ZZ$, there are only finitely many 
max-compatible order ideals associated with $M$. 
\label{finiti-maxC}
\end{prop}

To conclude this section, we sketch here an algorithm which allows 
to compute all the maximal compatible order ideals of a given module $M$.
It can be summarized as follows:

\begin{alg}[Computation of maximal compatible order ideals] \mbox{}\\
\textsc{Input:} A module (lattice) $M\subseteq \ZZ$ given 
by a finite set of generators.\\
\textsc{Output:} All the maximal compatible order ideals w.r.t.\ $M$.
\begin{description}
\item{\emph{Step 1.}} Compute the set $\mathcal{A}_1$ of 
minimal elements of $\mathcal{A}$ w.r.t.\ the partial order~$\preceq$;
\item{\emph{Step 2.}} Let $\V = \NN \setminus 
\bigcup_{a\in \mathcal{A}_1} C(a)$;
\item{\emph{Step 3.}} Construct the set $\mathcal{X}_1$ of the minimal 
elements of
$M \setminus \left(P \cup -P \cup \{0\}\right)$ 
w.r.t.\ $\sqsubseteq$; 
\item{\emph{Step 4.}} Define the partition on $\V$ as in (\ref{partizione})
and \ref{partizOperativa};
\item{\emph{Step 5.}} Construct the graph $\widetilde{G}_\V$ whose 
vertexes are the elements of the above partition and two vertexes $R_u$ and 
$R_v$ are not connected if and only if there exists $a\in \mathcal{X}_1$ such
that $a_0 \in D(u)$ and $a_1 \in D(v)$ (see proposition~\ref{nonConnessione});
\item{\emph{Step 6.}} Compute the 
maximal cliques of~$\widetilde{G}_\V$;
\item{\emph{Step 7.}} Recover the maximal cliques of~$G_\V$: if $R_{u_1}, \dots, 
R_{u_k}$ is a maximal clique of~$\widetilde{G}_\V$, then the corresponding
maximal clique of~$G_\V$ is $R_{u_1}\cup \cdots \cup R_{u_k}$.
\item{\emph{Step 8.}} Return all the maximal cliques computed in Step 7.
\end{description}
\end{alg}

\begin{rem}
The computation of $\mathcal{A}_1$ in step 1 can be done in a finite 
number of steps (according to 
remark~\ref{oss1}, the problem is equivalent to the problem of finding 
a minimal set of generators of a monomial ideal; one way to proceed, is
suggested in the proof of proposition~\ref{x1}); the set $\V$ of step 2
can be infinite, but is described in finite terms; a possible construction
of $\mathcal{X}_1$ is given again in the proof of proposition~\ref{x1}; 
to get 
the partition of $\V$ in step 4 it is enough to find elements 
$u\in \V$ such that the $l$-tuples $s(u)$ assume all the possible (finite)
values.
\end{rem}

\begin{exmp} 
\label{ex1111@}
We consider again the module of example~\ref{ex1}.
Since in $M$ we have the elements $(-2, 4)$ and $(6, -2)$, in $\mathcal{X}$
we have, among others, the four elements $((2, 0), (0, 4))$, $((6, 0), (0, 2))$
and $((0, 4), (2, 0))$, $((0, 2), (6, 0))$. It is easy to see that these
four elements are all the elements of $\mathcal{X}_1$.\\
Starting from $\mathcal{X}_1$ we can 
divide the set $\V$ (showed in figure~\ref{figA}, right) into six regions,
according to~(\ref{partizione}). The six regions
are $R_{(0, 0)}$, $R_{(0, 2)}$, $R_{(0, 4)}$, $R_{(2, 0)}$, 
$R_{(2, 2)}$, $R_{(6, 0)}$ and in figure~\ref{fig1B} are labeled, respectively, 
by $A, B, C, D, E$ and $F$. 
Hence $\widetilde{G}_\V$ has 6 vertexes and the edges that
are not connected (according to proposition~\ref{nonConnessione}) are: 
$BF, CD, CE, CF, EF$. The maximal cliques of 
$\widetilde{G}_\V$ are $(A, B, C)$, $(A, B, D, E)$ and $(A, D, F)$. 
They correspond to the maximal cliques of $G_\V$ which are 
$(A\cup B \cup C)$,
$(A \cup B \cup D \cup E)$ and $(A\cup D\cup F)$. These sets 
(each of 20 elements) are all the maximal compatible order ideals 
(w.r.t.\ $M$) and are all maximum (note that $\widetilde{G}_\V$ has only 
one maximum clique, which is $(A, B, D, E)$, hence it is evident that
maximum cliques of $G_\V$  in general do not correspond to 
maximum cliques of $\widetilde{G}_\V$). 
\end{exmp}

\begin{figure}
\resizebox{!}{4.5cm}{
\includegraphics{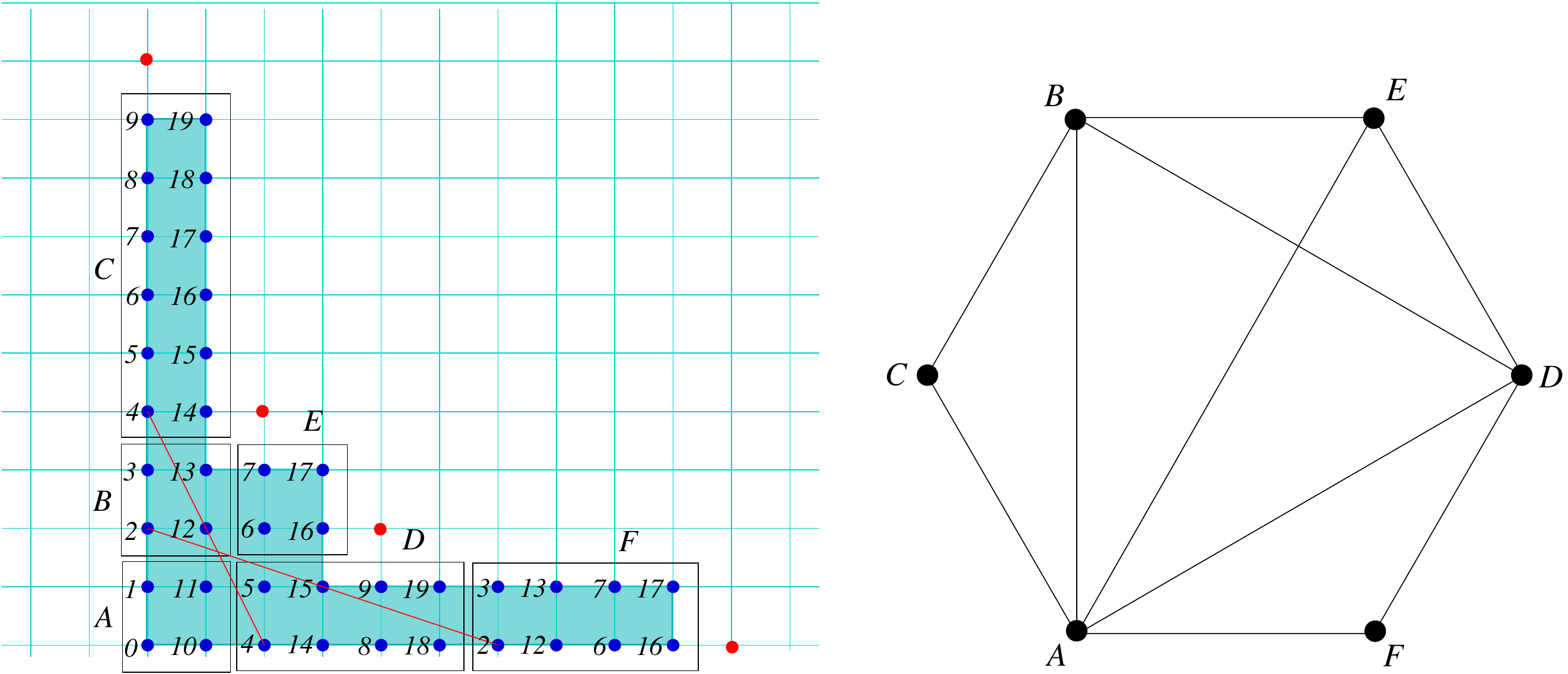}
}
\caption{On the left: the lattice of example~\ref{ex1} and the set $\V$ 
(in light blue), whose elements are labeled
according to~(\ref{numB}) and divided into the regions $A$, $B$, \dots,  
$F$. The two red segments represent the elements of $\mathcal{X}_1$.
On the right we have the graph $\widetilde{G}_\V$, whose vertexes are 
the regions $A, B, \dots, F$.}
\label{fig1B}
\end{figure}

\begin{exmp} 
\label{ex2}
Let $M = \langle (2, 1, 4), (0, 3, -3)\rangle\subseteq 
\mathbb{Z}^3$. $M$ is a module of rank 2 in $\mathbb{Z}^3$. The set 
$\mathcal{A}_1$ is 
$\{(0, 3, 3), (2, 1, 4), (2, 4, 1), (6, 0, 15), (6, 15, 0)\}$. 
The set $\V$ is
therefore the complement (in $\mathbb{N}^3$) of the cones $C(a)$ where
$a\in \mathcal{A}_1$. The set $\mathcal{X}_1$ is:  
\[
\begin{array}{c}
((4, 11, 0), (0, 0, 1)),\ ((0, 3, 0), (0, 0, 3)),\ ((4, 0, 11), (0, 1, 0)),\\
((2, 0, 7), (0, 2, 0)),\ ((2, 7, 0), (0, 0, 2))
\end{array}
\]
and 5 other couples obtained inverting the above couples. Therefore the
set $\mathcal{Y}$ is:
\[ 
\begin{array}{c}
(0, 0, 0), (4, 11, 0), (0, 0, 1), (0, 3, 0), (0, 0, 3), (4, 0, 11),\\
 (0, 1, 0), (2, 0, 7), (0, 2, 0), (2, 7, 0), (0, 0, 2)
\end{array}
\]
and we get a partition of $\V$ into $19$ classes $R_u$, 
where $u$ is one of the following points of $\mathbb{Z}^3$:
\[
\begin{array}{c}
(0, 0, 0), (0, 0, 1), (0, 0, 2), (0, 0, 3), (0, 1, 0), (0, 1, 1), (0, 1, 2), \\
(0, 1, 3), (0, 2, 0), (0, 2, 1), (0, 2, 2), (0, 2, 3), (0, 3, 0), \\
(0, 3, 1), (0, 3, 2), (2, 0, 7), (2, 7, 0), (4, 0, 11), (4, 11, 0).
\end{array}
\]
For instance, $R_{(0,0,0)}$ is the hyper-rectangle 
$\{(i, 0, 0) \mid i \in \mathbb{N}\}$, 
while $R_{(4, 0, 11)}$ is given by the union of two hyper-rectangles:
\[
R_{(4, 0, 11)} = \left\{
(i, 0, h) \mid i\geq 4, \ 11\leq h \leq 14
\right\} \cup \left\{(i, 0, h) \mid 4 \leq i \leq 5,\ h\geq 15
\right\}.
\]
The graph $\widetilde{G}_\V$ has 19 vertexes; the computation of the 
maximal cliques gives 6 elements, hence the module $M$ has 6 maximal order 
ideals. An 
example of a maximal clique is given by:
$R_{(0,0,0)}$, $R_{(0, 0, 1)}$, $ R_{(0, 0, 2)}$, $ R_{(0, 0, 3)}$, $ R_{(2, 0, 7)}$, 
$R_{(4, 0, 11)}$ and the union of these sets gives the corresponding
maximal order ideal $\oo = H_1 \cup H_2$, where
\[
H_1 = \{(i, 0, j)\mid i\geq 0, \ 0 \leq j \leq 14\}, \quad
H_2 = \{(i, 0, j) \mid 0 \leq i \leq 5, \ j \geq 15 \}.
\]
\end{exmp}
In general it is not true that maximal compatible order ideals are 
max-compatible. Here is an example in the case of a rank 3 module in 
$\mathbb{Z}^3$: let 
$M = \langle (1, 1, 2), (0, 3, 1), (0, 0, 4)\rangle \subseteq \mathbb{Z}^3$;
the above algorithm gives 23 maximal compatible order ideals, 19 of 12
elements (maximum), 2 of 9 elements and 2 of 8 elements. In particular:
$D(2, 0, 0) \cup D(0, 2, 0) \cup D(0, 0, 3)$ is a compatible order ideal 
with 8 elements which is maximal but not maximum.

When $\mathrm{rank}\; M = n$, the max-compatible order ideals, as shown in the 
above example, can easily be selected counting their elements. When the 
rank of $M$ is less than $n$, 
it is necessary to have another criterion to select, from the maximal 
order ideals, the max-compatible ones. 

One possible way to proceed is to check if $\rho(\oo) = B$, where the 
map $\rho$ is described in section~\ref{sez2} (and is obtained by successive 
divisions by the rows of $M$). The order
ideal $\oo$ is a finite union of $R_u$'s, hence $\oo$ is a finite union 
of hyper-rectangles by proposition~\ref{hyperrect}.
It is possible to show that, using the pivot elements of the matrix 
$M$, the image under $\rho$ of a hyper-rectangle can be decomposed 
into a finite union of 
sets of points of the form $(F_1(i_1, \dots, i_r),  \dots,
F_n(i_1, \dots, i_r))$, where $F_1, \dots, F_n$ are linear functions and 
$i_1, \dots, i_r$ are integer numbers
bounded by suitable inequalities. From this it follows that the 
image under the map $\rho$ of a hyper-rectangle can be described in finite
terms and the check $\rho(\oo) = B$ can be done algorithmically. We note, 
moreover, that the sketched construction allows to also obtain the map 
$\sigma : B \longrightarrow \oo$ which is the inverse of~$\rho$.

An example can clarify this construction: take the order ideal $\oo$
considered in example~\ref{ex2}, which is the union of the hyper-rectangles 
$H_1$ and $H_2$ and let us see how to compute $\rho(H_1)$ and $\rho(H_2)$. 
Since the pivot element of the first row of $M$ is 2, in order to compute 
the reduction of $H_1$ by the matrix $M$, we decompose 
$H_1$ into two disjoint parts: 
$H_1 = \{(2h, 0, j) \} \cup \{(2h+1, 0, j) \}$, ($h \geq 0$). 
The elements of the form $(2h, 0, j)$ are reduced w.r.t.\ the first row 
of $M$ to $(0, -h, j-4h)$ (while $(2h+1, 0, j)$ reduces to $(1, -h, j-4h)$).
The pivot element of the second row of $M$ is $3$, so we consider three 
cases: $h=3l$, $h=3l+1$, $h=3l+2$ ($l\geq 0$). For instance, in case
$h = 3l$,  the elements $(0, -h, j-4h)$ become $(0, -3l, j-12l)$ 
which reduce to $(0, 0, j-15l)$. In this way we can see that 

\[
\rho(H_1) = \{(i, 0, j) \mid j \leq 14\} \cup \{(i, 1, j) \mid j \leq 3 \}
\cup \{(i, 2, j) \mid  j \leq 7\}
\]
and, similarly, 
\[
\rho(H_2) = \{(i, 0, j) \mid j \geq 15\} \cup \{(i, 1, j) \mid j \geq 4 \}
\cup \{(i, 2, j) \mid  j \geq 8\}
\]
where $0 \leq i \leq 1$, and it is clear that $\rho(H_1) \cup \rho(H_2) = B$,
since $B = \{(i, j, h) \mid 0 \leq i \leq 1, \ 0 \leq j \leq 2, \ h \in 
\mathbb{Z}\}$.\\
Repeating these constructions for all the order ideals
obtained in example~\ref{ex2}, it is possible to verify that all these 
order ideals are indeed max-compatible.

\begin{prop}
If $\oo$ is a max-compatible order ideal and $z\in \ZZ$, then it is possible
to compute an element $b \in \oo$ such that $b \equiv_M z$.
\label{repr}
\end{prop}

\begin{proof} According to \emph{Step 7} of the previous algorithm, the 
order ideal $\oo$ is a union of sets of the form $R_u$, where each $R_u$
is a finite union of hyper-rectangles (proposition~\ref{hyperrect}), hence it 
is enough to solve the following problem: given a hyper-rectangle $R$ and 
an element $z \in \ZZ$, check if there exists $b\in R$ such that 
$b \equiv_M z$. It is easy to see that this problem can be converted into
the problem of solving a set of linear Diophantine 
equations. 
\end{proof}

\section{Border bases} 
\label{bb}
If $I_M$ is a lattice ideal given by the lattice
$M\subseteq \ZZ$ and if $\oo$ is a max-compatible ($\mathrm{mod}\, M$)
order ideal of $\NN$, then the set $\esp(\oo)$ is an order ideal 
of $\kx$ and is a basis of $\kx/I_M$ as a $K$-vector space and we can define
the $\esp(\oo)$-border basis of $I_M$, which is given by:
\begin{eqnarray}
G_\oo = \{\esp(b) - \esp(\bar{b}) \ | \ \mbox{for } b \in \partial(\oo)\}
\label{bbase}
\end{eqnarray}
where $\bar{b}$ is the representative of $b$ in $\oo$. Since a border
basis is constructed starting from a max-compatible order ideal, as 
a consequence of proposition~\ref{finiti-maxC} we have:
\begin{prop}
The number of border bases of any lattice ideal $I_M$ (where 
$\mathrm{rank}\, M \leq n$) is finite.
\end{prop}

Notice that the situation here parallels that of
Gr\"obner bases, where
the number of all possible reduced Gr\"obner bases of an ideal is finite
(see~\cite{mrob}) and can be read off from the construction of 
the Gr\"obner fan (see also~\cite{bfs, stur}).

Regarding the computation of (all) the border bases of a lattice ideal 
$I_M$, we have that if $\mathrm{rank}\, M=n$, then any max-compatible order
ideal is finite, hence~(\ref{bbase}) gives the border basis $G_\oo$ as a finite
set. If $\mathrm{rank}\, M < n$, then the border of a max-compatible order 
ideal $\oo$ is infinite. As a consequence of section~\ref{coi}, $\oo$ is
a finite union of hyper-rectangles contained in $\NN$, hence its border is 
contained in the borders of hyper-rectangles. 

Assume 
$R = \{(a_1, \dots, a_n) \ | \ l_i \leq a_i < L_i\}$ is one of the 
hyper-rectangles of the decomposition of $\oo$ (moreover, 
it is not restrictive to 
assume $l_i = 0$, since $\oo$ is an order ideal). For each $j$ such that
$L_j \not= +\infty$ we consider the elements:
\[
\{(a_1, \dots, L_j, \dots, a_n) \ | \ l_i \leq a_i < L_i, \ i \not= j\}.
\]
These hyper-rectangles give the border of $R$. In this way we eventually
describe the border of $\oo$ (as a finite union of hyper-rectangles). 
To obtain the representatives of the elements of $\partial \oo$ in $\oo$
(and hence to obtain the $\oo$-border basis), it is enough to compute
$\sigma(\rho(b))$ for each $b \in \partial \oo$ (or to use 
proposition~\ref{repr}). Although $\partial \oo$ is infinite, its 
description in terms of finite hyper-rectangles allows to describe
the $\oo$-border basis in finite terms.
Here we explain this construction with an example.

\begin{exmp}
\label{ex3}
Let us take again the module $M$ considered in example~\ref{ex2} and 
in particular the order ideal $\oo\subseteq \NN$ defined in there. 
The corresponding
lattice ideal $I_M$ is $(y-x^4z^{11}, x^6z^{15} - 1)$. The border of $\oo$
is the union of the following sets: $B_1 = \{(6+p, 0, 15) \mid p \geq 0 \}$, 
$B_2 = \{(6, 0, 16+p) \mid p \geq 0 \} $, 
$B_3 = \{(p, 1, 15+q) \mid 0 \leq p \leq 5, \ 
q \geq 0 \} $ and $B_4 = \{(p, 1, q) \mid p \geq 0, \ 0 \leq q \leq 14 \}$.
Finally, each element of $\partial\oo$ has a representative in $\oo$ 
according to the following scheme:
\medskip

\begin{center}
\begin{tabular}{|c|c||c|c|} \hline 
$b\in \partial \oo$ &\rule{0mm}{4mm} $\bar{b}\in \oo$ & $b\in \partial \oo$ & 
$\bar{b}\in \oo$\\ \hline
$(6+i, 0, 15)$ & $(i, 0, 0)$ & $(6, 0, 16+i)$ & $(0, 0, 1+i)$\\ 
$(j, 1, 15+i)$ & $(4+j, 0, 26+i)$ & $(2+h, 1, 15+i)$ & $(h, 0, 11+i)$ \\
$(i, 1, h)$ & $(4+i, 0, 11+h)$ & $(j, 1, 4+k)$ & $(4+j, 0, 15+j)$ \\
$(2+i, 1, 4+k)$ & $(i, 0, k)$ & & \\ \hline
\end{tabular}
\end{center}

\medskip
\noindent
where $i \geq 0$, $0 \leq j \leq 1$, $0 \leq h \leq 3$ and 
$0 \leq k \leq 10$.
\end{exmp}
Note, however, that if $I_M$ is a lattice ideal and if $\esp(\oo)$ is
an order ideal (assume $\oo$ max-compatible), to express any $f \in \kx$ 
as a linear combination of monomials in $\esp(\oo)$, it is not necessary
to have the $\esp(\oo)$-border basis and then use a reduction as is usually
done in the general case (see~\cite{kr}, proposition~6.4.11). As a 
consequence of proposition~\ref{repr}, any monomial of $f$ can directly be 
reduced to a monomial of $\esp(\oo)$, avoiding the construction of the 
border basis. 

It is well known that any Gr\"obner basis of an ideal $I$ gives an
order ideal in $\kx$ which is a $K$-basis of $\kx/I$ and, consequently, 
a border basis of $I$ (for instance, the order ideal considered in 
example~\ref{ex3} comes from the Gr\"obner basis of $I_M$ computed 
w.r.t.\ the lex term-order in which $x < z < y$). 
It is also well known that in general 
there exist border bases which do not come from any Gr\"obner bases, hence 
it is interesting to see what can be said regarding the border bases and 
the Gr\"obner bases of a lattice ideal $I_M$. Although in many examples
border bases are indeed Gr\"obner bases (see next section), 
there are several cases of border bases of lattice ideals 
which cannot come from Gr\"obner bases. Here we give an example.

Let us consider the module $M\subseteq \mathbb{Z}^3$ generated by 
the vectors $(1, -2, -1)$, $(1, -1, 2)$ and $(-2, -1, 1)$. The corresponding 
HNF for $M$ is:
\[
M_H = \left(
\begin{array}{rrr}
1 & 0 & 5\\
0 & 1 & 3 \\
0 & 0 & 14
\end{array}
\right). 
\]
The set $\V$ is given by $\mathbb{N}^3 \setminus \cup_{a\in \mathcal{A}_1}C(a)$, 
where $\mathcal{A}_1 = \{(0, 0, 14), (1, 0, 5), (0, 14, 0)$,   
$(0, 4, 2)$, $(14, 0, 0)$, $(2, 1, 1)$, $(0, 5, 1)$, $(1, 2, 1)$, 
$(0, 1, 3)$, $(1, 3, 0)$, $(4, 2, 0)$, $(3, 0, 1)$, $(1, 1, 2)$, 
$(2, 0, 4)$, $(5, 1, 0) \}$, hence can be described by:
\[
\V = \bigcup_{i=1}^{10} D(P_i)
\]
where $P_1=(13, 0, 0)$, $P_2=(0, 13, 0)$, $P_3=(0, 0, 13)$, 
$P_4=(4, 1, 0)$, $P_5=(3, 2, 0)$, $P_6=(2, 0, 3)$, $P_7=(1, 0, 4)$, 
$P_8=(0, 3, 2)$, $P_9=(0, 4, 1)$, $P_{10}=(1, 1, 1)$.

The binomial ideal $I_M$ associated with $M$ is: 
\[
I_M = \left( xz^2 - y, y^2z - x, y^3 - x^2z, xy^2 - z^3, x^2y - z, 
x^3 - yz^2, z^4 - x^2, yz^3 - 1 \right).
\]
The computation of the Gr\"obner fan of $I_M$ (one possibility is to 
use the implementation presented in Sage, see~\cite{sage}) shows 
that $I_M$ has
33 reduced Gr\"obner bases, while the computation of all the possible 
max-compatible order ideals of $I_M$ obtained with the techniques 
developed in section~\ref{coi} gives 35 order ideals. Hence there must be 
two max-compatible order ideals (in $\mathbb{N}^3$) which do not come from  
Gr\"obner bases. They are the following: 
\[
\oo_1 = D(1, 2, 0) \cup D(2, 0, 1) \cup D(0, 1, 2) \cup D(1, 1, 1)
\]
and
\[
\oo_2 = D(3, 0, 0) \cup D(0, 3, 0) \cup D(0, 0, 3) \cup D(1, 1, 1).
\]
They both have 14 elements (according either to the determinant of
the matrix $M_H$ or to the 
dimension of $K[x, y, z]/I_M$ as a $K$ vector space). 
The border basis corresponding to $\esp(\oo_1)$ contains the following three 
binomials:
\[
x^3-yz^2, \quad y^3-x^2z, \quad z^3-xy^2
\]
(where $x^3, y^3$ and $z^3$ are in the border of $\esp(\oo_1)$, while 
$yz^2, x^2z$ and $xy^2$ are in $\esp(\oo_1)$). If $\esp(\oo_1)$ were 
an order ideal coming from a Gr\"obner basis corresponding to a term order
$<_\sigma$, then we would have: $x^3>_\sigma yz^2,\ y^3>_\sigma x^2z,\ 
z^3 >_\sigma xy^2$ and these conditions are not compatible. \\
A similar contradiction can be found with the order ideal $\oo_2$. \\
It is possible to verify that the matrix $M_H$ of this example 
is minimal, in the sense that any other matrix of the form
\[ 
\left(
\begin{array}{rrr}
1 & 0 & a\\
0 & 1 & b \\
0 & 0 & c
\end{array}
\right)
\] 
where $a \leq 5$, $b \leq 3$ and $c \leq 14$, $(a, b, c) \not = (5, 3, 14)$,
corresponds to an ideal in which 
all border bases come from Gr\"obner bases.  

\section{The case $m = n = 2$}
\label{n=2}
The case of a two dimensional module $M$ in $\mathbb{Z}^2$ is 
particularly simple. In this section we show the main points.
We consider the following two matrices obtained from the generators of $M$:
\begin{eqnarray}
\left(
\begin{array}{cc}
a_1 & a_2 \\
0 & a_3
\end{array}
\right), \qquad 
\left(
\begin{array}{cc}
b_1 & b_2 \\
b_3 & 0
\end{array}
\right).
\label{matr2x2}
\end{eqnarray}
The first is the usual matrix in HNF (hence $a_1 > 0$, 
$0\leq a_2 < a_3$), the second is in HNF with respect to 
the second and first column, hence we have $b_2 > 0$, $0 \leq b_1 <b_3$;
moreover, the relations between the $a$'s and $b$'s are:
$b_2 = \mathrm{gcd}(a_2, a_3)$, $b_3 = a_1a_3/b_2$ and 
$b_1 = a_1\cdot \min \{\lambda \in \mathbb{N} \ | \ \exists \mu \in \mathbb{Z}: 
b_2 = \lambda a_2 + \mu a_3 \}$.

Let $\mathcal{B}_1$ be the set of minimal elements of $M \cap \mathbb{N}^2
\setminus \{0\}$ w.r.t.\ $\preceq$ and $\mathcal{B}_2$ be the set of 
minimal elements of $\{(p, q) \in \mathbb{N}^2 \ | \ (p, -q) \in M 
\setminus \{0\}\}$ again w.r.t.\ $\preceq$. Both $\mathcal{B}_1$ and 
$\mathcal{B}_2$ contain the elements $(b_3, 0)$ and $(0, a_3)$. 
The set $\mathcal{A}_1$ of proposition~\ref{insiemeV} is contained in 
$\mathcal{B}_1  \cup \mathcal{B}_2$ hence (by 
proposition~\ref{insiemeV}), we have:
\[
\V = \mathbb{N}^2 \setminus \left(\bigcup_{P_1 \in \mathcal{B}_1} C(P_1) \cup
\bigcup_{P_2 \in \mathcal{B}_2} C(P_2)\right).
\]
We call $P, Q \in \mathcal{B}_2$ 
\emph{consecutive} if 
\[
\mathcal{B}_2 \cap D(\mathrm{lcm}(P, Q)) = \{P, \ Q\}.
\label{consecutivi}
\]
\begin{prop} Let $P =(p_1, p_2)$ and $Q=(q_1, q_2)$ be two consecutive 
elements of $\mathcal{B}_2$, let $R\in M$ be such that 
$R = (p_1 -q_1, q_2 - p_2)$ (if $p_1 \geq q_1$
and $q_2 \geq p_2 $) or $R = (q_1-p_1, p_2-q_2)$ (if $q_1 \geq p_1$
and $p_2 \geq q_2 $) 
and let $\oo(P, Q) = D(\mathrm{lcm}(P, Q)) \setminus C(R)$. Then
we have:
\begin{enumerate}
\item $\oo(P, Q)$ is a compatible order ideal;
\item $\oo(P, Q)$ has $a_1 a_3$ elements (hence is max-compatible); 
\item If $\oo$ is any compatible order ideal, then there exist $P, Q \in 
\mathcal{B}_2$ consecutive, such that $\oo \subseteq \oo(P, Q)$. 
\end{enumerate}
\label{propOO}
\end{prop}

\begin{figure}
\begin{center}
\resizebox{4cm}{!}{
\includegraphics{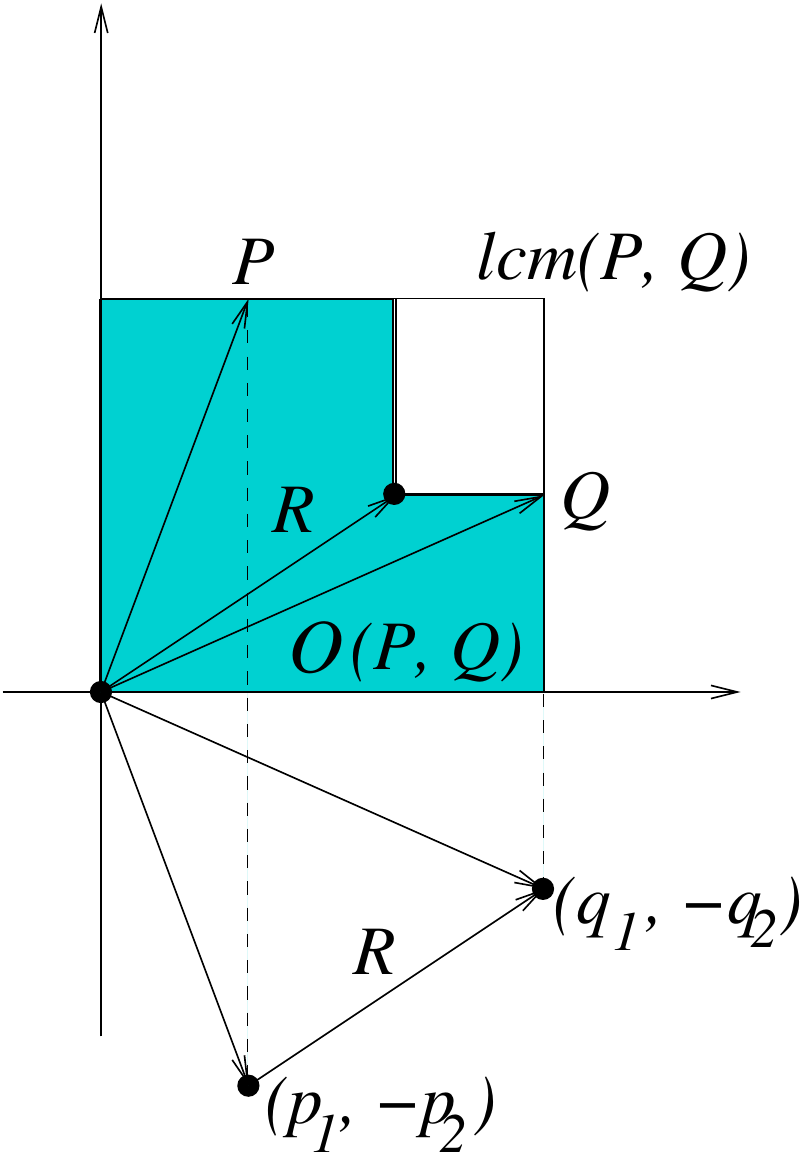}
}
\end{center}
\caption{The construction of the order ideal $\oo(P, Q)$ of 
proposition~\ref{propOO}.}
\label{figCaso2}
\end{figure}

\begin{proof} Clearly $\oo(P, Q)$ is an order ideal. If $A, B \in \oo(P, Q)$
are equivalent, we can assume that $A = (\alpha, 0)$ and $B = (0, \beta)$. 
The vector $(\alpha, -\beta)$ is an element of $M$ and it is easy to see 
that $(\alpha, \beta)$ is minimal w.r.t.\ $\preceq$, so $(\alpha, \beta) \in 
\mathcal{B}_2$, in contradiction with the consecutivity of $P$ and $Q$. 
Again, since $P$ and $Q$ are consecutive, in the parallelogram whose 
vertexes are $O$, $(p_1, -p_2)$, $(q_1, -q_2)$ and $(p_1, -p_2)+(q_1, -q_2)$
there are no other points of $M$, hence $M$ is generated by 
$(p_1, -p_2)$, $(q_1, -q_2)$ and therefore the determinant of the matrix:
\[
\left(
\begin{array}{cc}
p_1 & -p_2 \\
q_1 & -q_2
\end{array}
\right)
\]
(which is $p_2q_1 - p_1 q_2$) must be equal (in absolute value) to $a_1a_3$. 
A direct computation of the number of elements with integer coordinates 
contained in $\oo(P, Q)$ gives 
$|p_2q_1 - p_1 q_2|$, hence $\oo(P, Q)$ has $a_1a_3$ elements. In particular
$\oo(P, Q)$ is maximum. Finally, take any compatible order ideal $\oo$. Let
$P = (p_1, p_2), Q = (q_1, q_2) \in \mathcal{B}_2$ be consecutive, such that 
$(p_1, 0) \in \oo$ but $(q_1, 0) \not \in \oo$ ($P$ and $Q$ can always 
be found, since $\mathcal{B}_2$ contains 
$(b_3, 0)$ and $(0, a_3)$). 
Then clearly $\oo \subseteq \oo(P, Q)$. 
\end{proof}

As a consequence of the above proposition, we see that all the 
possible maximum compatible order ideals of $M$ are of the form
$\oo(P, Q)$, where $P, Q \in \mathcal{B}_2$ are consecutive. 
It is easy to verify that $(a_1, a_3-a_2)$ and $(0, a_3)$ are 
elements of $\mathcal{B}_2$ and are consecutive, analogously 
$(b_3-b_1, b_2)$ and $(b_3, 0)$ are also in $\mathcal{B}_2$ and are
consecutive. From them we get the following two maximum, compatible 
order ideals: $\{(\alpha, \beta) \ | \
0 \leq \alpha < a_1, \ 0 \leq \beta < a_3  \}$ and 
$\{(\alpha, \beta) \ | \ 0 \leq \alpha 
<  b_3, \ 0 \leq \beta < b_2\}$ (which are two rectangles).
If now $P, Q \in \mathcal{B}_2$ are any
consecutive elements (different from the two couples considered above), 
neither $P$ nor $Q$ lies on one of the two coordinate axes, hence the 
order ideal $\oo(P, Q)$ is given by the difference of two rectangles, as 
in figure~\ref{figCaso2}. In conclusion we have:

\begin{prop}
Let $M \subseteq \mathbb{Z}^2$ be a 
two dimensional sub-module generated by the rows of the first and hence 
also the second matrix in~(\ref{matr2x2}). 
Then the maximal compatible order ideals 
of $M$ are: 
\begin{enumerate}
\item $\{(\alpha, \beta) \ | \
0 \leq \alpha < a_1, \ 0 \leq \beta < a_3  \}$;
\label{uno1}
\item $\{(\alpha, \beta) \ | \ 0 \leq \alpha 
<  b_3, \ 0 \leq \beta < b_2\}$ (where $b_2 = \mathrm{gcd}(a_2, a_3)$ 
and $b_3 = a_1a_3/b_2$);
\label{due2}
\item a difference of two rectangles: $D(\mathrm{lcm}(P,Q)) \setminus C(R)$,
where $P, Q \in \mathcal{B}_2$ are consecutive, with no zero coordinates
and $R$ is as defined in proposition~\ref{propOO}.
\label{tre3}
\end{enumerate}
\end{prop}
If $\oo$ is one of the order ideals described by the above proposition, 
then the \emph{corners} of $\esp(\oo)$ (as defined in~\cite{kr}, page~428) 
are either two elements (in case~(\ref{uno1}) and~(\ref{due2})) or three
(in case~(\ref{tre3})), as shown by the points $A$ and $B$ and $A$, $B$ and $C$
in figure~\ref{figAngoli}.
\begin{figure}
\begin{center}
\resizebox{!}{3cm}{
\includegraphics{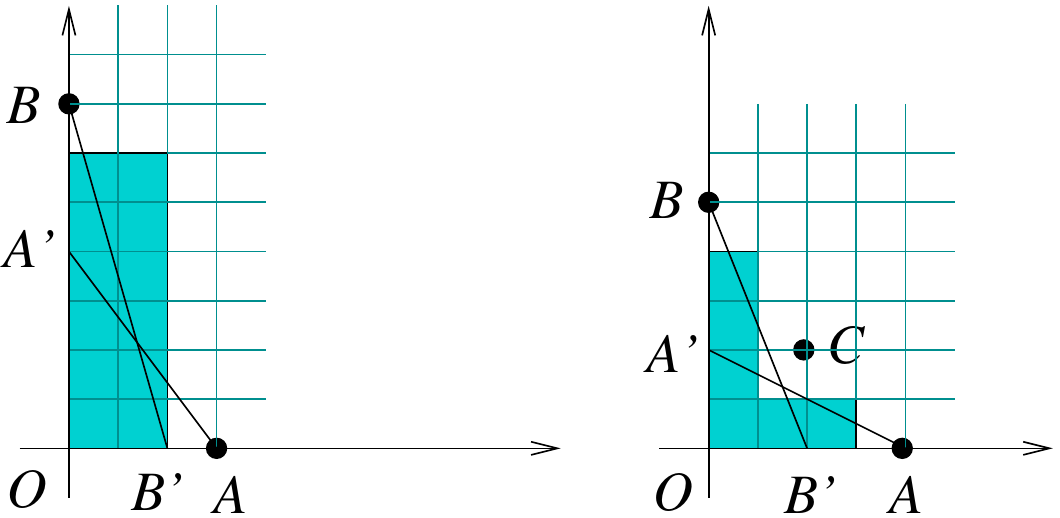}
}
\end{center}
\caption{The possible shapes of an order ideal for two dimensional lattices.}
\label{figAngoli}
\end{figure}
If $A$, $B$ are corners, let $A'$ and $B'$ denote their representatives 
in the order ideal (the representative of $C$ 
is necessarily $O$, since $C$ is an element of $M$). It is clear that
$A'$ and $B'$ have one coordinate 0 (if not, we could construct two equivalent
elements in the order ideal) and (recalling the characterizations of 
term-orders given in~\cite{rob}), any line through $O$ which has a 
slope between the slope of the line $BB'$ and the slope of $AA'$ gives 
rise to a term order $<_\sigma$ in $\mathbb{N}^2$ (and hence in $K[x, y]$) 
in which $A' <_\sigma A$, $B' <_\sigma B$ (and $0 <_\sigma C$). 
As a consequence of the 
above considerations and of~\cite{kr}, proposition~6.4.18, we have:

\begin{prop}
Let $M \subseteq \mathbb{Z}^2$ be as above and let $I_M$ be the corresponding
lattice ideal. Any maximal compatible order ideal w.r.t.\ $M$ corresponds
to the lattice ideal constructed from a Gr\"obner bases of $I_M$ (and 
conversely). $I_M$ has two reduced Gr\"obner bases of two elements
which are $\{x^{a_1}-1, \ x^{a_3}-1\}$ and $\{x^{b_3}-1, \ x^{b_2}-1\}$ and 
all the other reduced Gr\"obner bases have three elements 
of the form $x^\alpha-y^{\alpha'}$, $y^\beta-y^{\beta'}$, $x^{\gamma_1}y^{\gamma_2}-1$, 
where $A = (\alpha, 0)$, $A' = (0, \alpha')$, $B = (0, \beta)$, 
$B' = (\beta', 0)$, $C = (\gamma_1, \gamma_2)$, $A, B, C$ are 
corners of an order ideal, $A', B'$ are the representative of $A$ and $B$ 
in the order ideal.
\end{prop}

\bibliographystyle{plain}
\bibliography{boffi_logar_bibfile}

\end{document}